\documentclass{amsart}
\usepackage{amssymb,amsmath,amsfonts,amsthm,graphicx,enumerate,amscd,latexsym,curves,multibox}
\usepackage{bbm}
\usepackage{mathrsfs}
\usepackage[mathscr]{eucal}
\usepackage{epsfig,epsf,xypic,epic}
\usepackage{color}

\usepackage[T1]{fontenc}
\usepackage[sc]{mathpazo}

\DeclareGraphicsExtensions{.pdf,.png,.jpg,.bmp}

\def\bZ{\textbf{Z}}

\def\bR{\textbf{R}}
\def\bC{\textbf{C}}
\def\bP{\textbf{P}}

\newtheorem{thm}{Theorem}[section]

\newtheorem{defn}{Definition}[section]
\newtheorem{nb}{Remark}[section]
\newtheorem{conj}{Conjecture}[section]

\numberwithin{equation}{section}

\begin{document}

\title[Matrix factorizations from SYZ]{Matrix factorizations from SYZ transformations}
\author[K.-W. Chan]{Kwokwai Chan}
\address{Department of Mathematics, Harvard University, Cambridge, MA 02138}
\email{kwchan@math.harvard.edu}
\author[N.-C. Leung]{Naichung Conan Leung}
\address{The Institute of Mathematical Sciences and Department of Mathematics, The Chinese University of Hong Kong, Shatin, Hong Kong}
\email{leung@math.cuhk.edu.hk}

\dedicatory{Dedicated to Prof. S.-T. Yau on the occasion of his 60th birthday.}

\begin{abstract}
It is known that Lagrangian torus fibers of the moment map of a toric Fano manifold $X$, equipped with flat $U(1)$-connections, are mirror to matrix factorizations of the mirror superpotential $W:\check{X}\rightarrow\bC$. Via SYZ mirror transformations, we describe how this correspondence, when $X$ is $\bP^1$ or $\bP^2$, can be explained in a geometric way.
\end{abstract}

\maketitle

\tableofcontents

\section{Introduction}\label{sec1}

In \cite{Chan-Leung08a}, \cite{Chan-Leung08b}, we study the mirror symmetry correspondence between the symplectic geometry (A-model) of a toric Fano manifold $X$ and the complex geometry (B-model) of its mirror Landau-Ginzburg model $(\check{X},W)$ using SYZ transformations. In this sequel, we will continue to investigate this correspondence, but the theme this time is Homological Mirror Symmetry. More specifically, we are going to study the correspondence between Lagrangian torus fibers of the moment map of $X$ (equipped with flat $U(1)$-connections) and \textit{matrix factorizations} of the mirror superpotential $W:\check{X}\rightarrow\bC$ (cf. \cite{KL02}) via SYZ transformations.

In the toric Fano case, Homological Mirror Symmetry (\`{a} la Kontsevich) asserts that the derived Fukaya category $DFuk(X)$ of a toric Fano manifold $X$ is equivalent, as a triangulated category, to the category of matrix factorizations $MF(\check{X},W)$ of the mirror Landau-Ginzburg model $(\check{X},W)$:\footnote{More precisely, this is just "one half" of the Homological Mirror Symmetry conjecture. See Cho-Oh \cite{CO03} and Cho \cite{Cho04} for related results.}
\begin{equation}\label{hms1}
DFuk(X)\cong MF(\check{X},W).
\end{equation}
In \cite{Orlov03}, Orlov showed that $MF(\check{X},W)$ is equivalent to \textit{the category of singularities} $D_{Sg}(\check{X},W)$ of the Landau-Ginzburg model $(\check{X},W)$. Hence, the above Homological Mirror Symmetry statement can be reformulated as the following equivalence:
\begin{equation}\label{hms2}
DFuk(X)\cong D_{Sg}(\check{X},W).
\end{equation}

From the point of view of the SYZ conjecture \cite{SYZ96}, the equivalence (\ref{hms2}) appears to be much more natural. This is because Lagrangian torus fibers of $X$ equipped with flat $U(1)$-connections can be transformed (via SYZ transformations) to structure sheaves of points in $\check{X}$. From this, one can proceed to define the functor we need in (\ref{hms2}).

On the other hand, matrix factorizations are geometrically endomorphisms of holomorphic vector bundles. The SYZ construction tells us that holomorphic vector bundles should be mirror to Lagrangian \textit{multi-sections}. However, the equivalence (\ref{hms1}) says that Lagrangian torus fibers of $X$ equipped with flat $U(1)$-connections are also corresponding to matrix factorizations of $(\check{X},W)$. Apparently, this seems to be in conflict with the SYZ picture.

The aim of this paper is to demonstrate that the equivalence (\ref{hms1}) is in fact also compatible with the SYZ picture. We will describe how the correspondence between Lagrangian torus fibers of $X$ equipped with flat $U(1)$-connections and matrix factorizations of $(\check{X},W)$ can be explained geometrically using SYZ transformations. We shall consider the case when $X$ is either $\bP^1$ or $\bP^2$, equipped with a toric K\"{a}hler structure $\omega$, i.e. the Fubini-Study K\"{a}hler form. Let $P$ be the moment polytope associated with the Hamiltonian torus-action on $(X,\omega)$. Let $L_0$ be the Lagrangian torus fiber over the center of mass $x_0$ of the polytope $P$ and equip $L_0$ with the trivial flat $U(1)$-connection $\nabla$. This gives an A-brane $(L_0,\nabla)$ on $X$. We want to cook up a matrix factorization $M_0$ of $W:\check{X}\to\bC$, which is mirror to this A-brane, via SYZ mirror transformations.\footnote{Though we restrict our attention to these so-called \textit{Clifford tori} in projective spaces, our methods can in fact deal with other Lagrangian torus fibers (and equipped with other flat $U(1)$-connections) as well.}

In general, a matrix factorization of a Landau-Ginzburg model $(\check{X},W)$ is a square matrix $M$ of even dimensions with entries in the coordinate ring $\bC[\check{X}]$ and of the form
$$M=\left( \begin{array}{cc}
0 & F \\
G & 0 \end{array} \right)$$
such that
\begin{equation}\label{MF}
M^2=(W-\lambda)\textrm{Id}
\end{equation}
for some $\lambda\in\bC$ (cf. \cite{KL02}, \cite{Orlov03}). Geometrically, $M$ should be viewed as an odd endomorphism of a (trivial) $\bZ/2\bZ$-graded holomorphic vector bundle over $\check{X}$.

For example, when $X=\bP^1$, the mirror is given by a bounded domain $\check{X}\subset\bC^*$ together with the superpotential $W=z+\frac{q}{z}$, where $z$ is a coordinate on $\bC^*$ and $q\in\bR_{<1}$; and $M_0$ is given by
$$M_0=\left( \begin{array}{cc}
0 & z-\sqrt{q} \\
1-\frac{\sqrt{q}}{z} & 0 \end{array} \right).$$
When $X=\bP^2$, the mirror is given by a bounded domain $\check{X}\subset(\bC^*)^2$ together with the superpotential $W=z_1+z_2+\frac{q}{z_1z_2}$, where $z_1,z_2$ are coordinates on $(\bC^*)^2$ and $q\in\bR_{<1}$; and $M_0$ is given by
$$M_0=\left(\begin{array}{cccc}
0 & 0 & z_1-q^{1/3} & z_2-\frac{q^{2/3}}{z_1} \\
0 & 0 & -(1-\frac{q^{1/3}}{z_2}) & 1-\frac{q^{1/3}}{z_1}\\
1-\frac{q^{1/3}}{z_1} & -(z_2-\frac{q^{2/3}}{z_1}) & 0 & 0\\
1-\frac{q^{1/3}}{z_2} & z_1-q^{1/3} & 0 & 0 \end{array}\right).$$

To produce $M_0$ from $(L_0,\nabla)$, our strategy is to first try to deform $L_0$ by Hamiltonian symplectomorphisms to another Lagrangian subspace $L\subset X$ such that $L$ is a multi-section over a certain open set $U$ contained in the interior $\textrm{Int}(P)$ of $P$. For each point $x\in U$, let $L_x\subset X$ be the Lagrangian torus fiber over $x$. We then count the number of holomorphic disks $\varphi:(D^2,\partial D^2)\rightarrow(X,L\cup L_x)$,\footnote{More precisely, the upper (resp. lower) half of $\partial D^2$ is mapped to $L$ (resp. $L_x$). See Section \ref{sec2} for precise definitions.} and use these enumerative data to define a $2r\times2r$ matrix $\Psi_{L,\nabla}(x)$, where $2r=|L\cap L_x|$ is the number of intersection points. Letting $x$ vary in $U$ defines a matrix-valued function $\Psi_{L,\nabla}$ on $U$. The SYZ transformation (or fiberwise Fourier transform) of $\Psi_{L,\nabla}$ would then give a matrix factorization $M_0$ of $(\check{X},W)$, which is mirror to $(L_0,\nabla)$.

The geometry behind this procedure can be described as follows: Each term in the mirror superpotential $W$ (which is a Laurent polynomial) corresponds to a holomorphic disk in $X$ with boundary in a Lagrangian torus fiber $L_x$. Intuitively, \textit{a matrix factorization of $W$ is corresponding to cutting each of these disks into two halves by the deformed Lagrangian $L$}. From this perspective, each term of each entry in $M_0$ is corresponding to a disk $\varphi:D^2\to X$ with boundary in $L\cup L_x$. This explains why we need to count these holomorphic disks. In fact, according to the Floer theory developed by Fukaya-Oh-Ohta-Ono \cite{FOOO06}, the matrix factorization $M_0$ is expected to be mirror to the Floer differential $\mathfrak{m}_1:=\mathfrak{m}_1((L_x,\nabla_y),(L,\nabla))$ for the A-branes $(L_x,\nabla_y)$ and $(L,\nabla)$, where $\nabla_y$ is a flat $U(1)$-connection on $L_x$, for which we have the following formula analogous to (\ref{MF}):
$$\mathfrak{m}_1^2=W-W(z_0),$$
where $z_0\in\check{X}$ is the point mirror to $(L_0,\nabla)$.

The construction of the deformed Lagrangian $L$ and the classification of disks can be carried out easily when $X=\bP^1$ (see Section \ref{sec2} for details). In the $\bP^2$ case, however, we encounter serious difficulty in implementing all the steps. In particular, the deformed Lagrangian subspace $L$, constructed as a union of Lagrangian strata, is highly singular, and hence it is very hard to classify the holomorphic disks $\varphi:(D^2,\partial D^2)\rightarrow(X,L\cup L_x)$. Our way out is to look for pairs of paths in $L$ and $L_x$ which can possibly form the boundaries of holomorphic disks. This provides a heuristic way to count holomorphic disks, and we shall use this counting to define the matrix-valued function $\Psi_{L,\nabla}$ in the $\bP^2$ case. Now, our main result can be stated as follows.
\begin{thm}\label{main}
The SYZ transformation of $\Psi_{L,\nabla}$ gives a matrix factorization $M_0$ of $(\check{X},W)$, which is mirror to the A-brane $(L_0,\nabla)$ on $X$.
\end{thm}
We conjecture that when $L$ is smoothed out, our heuristic counting will give the genuine counting of holomorphic disks. Furthermore, the matrix factorization $M_0$ should coincide with the Floer differential $\mathfrak{m}_1=\mathfrak{m}_1((L_x,\nabla_y),(L,\nabla))$ for the A-branes $(L,\nabla)$ and $(L_x,\nabla_y)$, where $\nabla_y$ is a flat $U(1)$-connection on the trivial complex line bundle $\underline{\bC}$ over $L_x$. We shall give an informal argument to support this conjecture.

The organization of this paper is as follows. In Section \ref{sec2}, we go through the simple but illustrative example of $X=\bP^1$ in details, where we can easily classify all the holomorphic disks. In Section \ref{sec3}, we study the case of $X=\bP^2$ and give an argument to justify our heuristic counting.\\

\noindent\textbf{Acknowledgements.} Both of us are heavily indebted to Prof. Shing-Tung Yau for his guidance and kind support over the years. It is our great pleasure to dedicate this article to Prof. Yau on the occasion of his 60th birthday.

We would like to thank Yong-Geun Oh for numerous useful discussions during his visit to CUHK in the summer of 2009. Thanks are also due to Cheol-Hyun Cho and Ke Zhu for their helpful comments on Floer theory. The research of the first author (K.W.C.) was supported by Harvard University and the Croucher Foundation Fellowship. The research of the second author (N.C.L.) was partially supported by RGC grants from the Hong Kong Government.

\section{The toy example: $X=\bP^1$}\label{sec2}

We equip $X=\bP^1$ with the toric K\"{a}hler form $\omega$($=\omega_{\textrm{FS}}$, the Fubini-Study form) associated to the moment polytope $P=[0,t]\subset\bR$ where $t>0$. Let $\mu:X\to P$ be the moment map. The mirror Landau-Ginzburg model is given by
$$\check{X}=\{z\in\bC^*:q<|z|<1\}\subset\bC^*,\ W=z+\frac{q}{z},$$
where $z$ is a complex coordinate on $\bC^*$ and $q=e^{-t}$. Let $L_0$ be the Lagrangian torus fiber over the center of mass $x_0=t/2\in P$, so that $L_0$ is a great circle in $\bP^1$. Equipping $L_0$ with the trivial flat $U(1)$-connection $\nabla$, we obtain an A-brane $(L_0,\nabla)$ on $X=\bP^1$.

According to the SYZ conjecture, the mirror B-brane should be the structure sheaf of the point $e^{-x_0}=\sqrt{q}\in\check{X}$. A matrix factorization corresponding to this sheaf (which is a skyscraper sheaf supported at $\sqrt{q}\in\check{X}$), through the equivalence established by Orlov \cite{Orlov03}, is given by the $2\times2$ matrix
$$M_0=\left( \begin{array}{cc}
0 & z-\sqrt{q} \\
1-\frac{\sqrt{q}}{z} & 0 \end{array} \right).$$
It is easy to check that $M_0^2=(W(z)-W(\sqrt{q}))\textrm{Id}$.

Our goal is to demonstrate how the matrix factorization $M_0$ can be obtained directly by SYZ mirror transformations (or Fourier transform). This in turn shows that the equivalence (\ref{hms1}) makes a natural sense from the point of view of the SYZ conjecture.

The first step is to construct a Lagrangian $L\subset X$, which is Hamiltonian isotopic to $L_0$ and is a multi-section over a certain open interval $U\subset\textrm{Int}(P)=(0,t)$. We work with the symplectic Darboux coordinates $(x,u)$ on $X=\bP^1$, where $x\in(0,t)$ and $u\in\bR/2\pi\bZ$. This means that, in these coordinates, the complement of the toric boundary divisors $\bP^1\setminus\{0,\infty\}\subset\bP^1$ is realized as the quotient $T^*\textrm{Int}(P)/\bZ$ of the cotangent bundle of $\textrm{Int}(P)$ by the lattice $\bZ\subset T^*\textrm{Int}(P)$ of locally constant 1-forms, and the symplectic form $\omega$ restricted to $\bP^1\setminus\{0,\infty\}$ is given by the canonical one, i.e. $\omega|_{\bP^1\setminus\{0,\infty\}}=dx\wedge du$.

\begin{figure}[htp]
\centering
\includegraphics[scale=1.3]{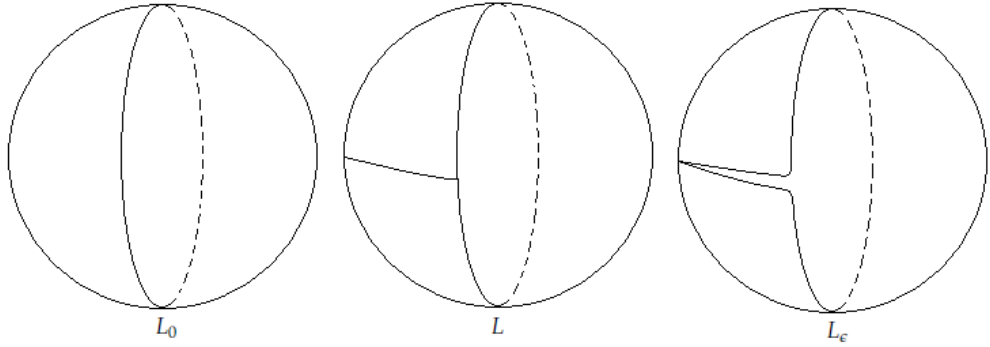}
\caption{The Lagrangians $L_0$, $L$ and $L_\epsilon$ in
$\bP^1$.}\label{CP^1}
\end{figure}

Define a map $\tau:[0,3]\rightarrow\bP^1$ by
\begin{equation*}
\tau(s)=\left\{\begin{array}{ll}
((1-s)t/2,0) & \textrm{for $0\leq s\leq1$;}\\
((s-1)t/2,0) & \textrm{for $1\leq s\leq2$;}\\
(t/2,2\pi(s-2)) & \textrm{for $2\leq s\leq3$.}
\end{array} \right.
\end{equation*}
This map gives a Lagrangian subspace in $\bP^1$, in the sense that $\tau^*\omega=0$. Denote this Lagrangian by $L$, which is Hamiltonian isotopic to $L_0$ because it cuts $\bP^1$ into two equal halves. Note that $L$ is the union of $L_0$ with two copies of the zero section of the moment map $\mu$ over $[0,t/2]$; it is singular and is not even immersed in $\bP^1$. However, we can deform $L$ to get an embedded Lagrangian submanifold $L_\epsilon\cong S^1$ in $\bP^1$. For example, this can be achieved by separating the two zero sections a little bit and then smoothing out the corners, in the way shown in Figure \ref{CP^1}. One has to be a bit careful in keeping the Hamiltonian isotopy class unchanged. Such a Lagrangian $L_\epsilon$ is of the form shown in the rightmost of Figure \ref{CP^1}.

Now, let $U=(0,t/2)\subset\textrm{Int}(P)=(0,t)$, and denote by $L_x$ the Lagrangian torus fiber over a point $x\in U$. For any $x\in U$, the Lagrangian $L$ (or more precisely, $L_\epsilon$) intersects $L_x$ in two points, which we label by $+$ and $-$. See Figure \ref{discs} below.

For $p,q\in L\cap L_x$, denote by $\pi_2(X;L,L_x;p,q)$ the set of homotopy classes of maps $\varphi:D^2\rightarrow X$ with $\varphi(\partial_+D^2)\subset L$, $\varphi(\partial_-D^2)\subset L_x$, $\varphi(-1)=p$ and $\varphi(1)=q$, where $D^2=\{z\in\bC:|z|\leq1\}$, $\partial D^2=\{z\in D^2:|z|=1\}$, $\partial_+D^2=\{z\in\partial D^2:\textrm{Im}(z)>0\}$ and $\partial_-D^2=\{z\in\partial D^2:\textrm{Im}(z)<0\}$. We will also denote such a map by $\varphi:(D^2,\partial D^2,-1,1)\rightarrow(X,L\cup L_x,p,q)$. Let $\pi_1(L_x;p,q)$ be the set of homotopy classes of maps $\gamma:[0,1]\rightarrow L_x$ such that $\gamma(0)=p$ and $\gamma(1)=q$; $\pi_1(L;p,q)$ is defined analogously. We have boundary maps
\begin{eqnarray*}
\partial_+:\pi_2(X;L,L_x;p,q)\rightarrow\pi_1(L;p,q),\\
\partial_-:\pi_2(X;L,L_x;p,q)\rightarrow\pi_1(L_x;p,q).
\end{eqnarray*}

We want to classify, for any $p,q\in L\cap L_x$, all maps $\varphi:(D^2,\partial D^2,-1,1)\rightarrow(X,L\cup L_x,p,q)$ which are nontrivial and holomorphic. More precisely, we shall look for maps which can be deformed to holomorphic disks $\varphi_\epsilon:(D^2,\partial D^2)\rightarrow(X,L_\epsilon\cup L_x)$ as $L$ is being smoothed to give $L_\epsilon$. In the $\bP^1$ case, it is not hard to see that there are totally four such holomorphic disks: two for $p=+,q=-$ and two for $p=-,q=+$. This is illustrated in Figure \ref{discs} below.

\begin{figure}[htp]
\centering
\includegraphics[scale=0.8]{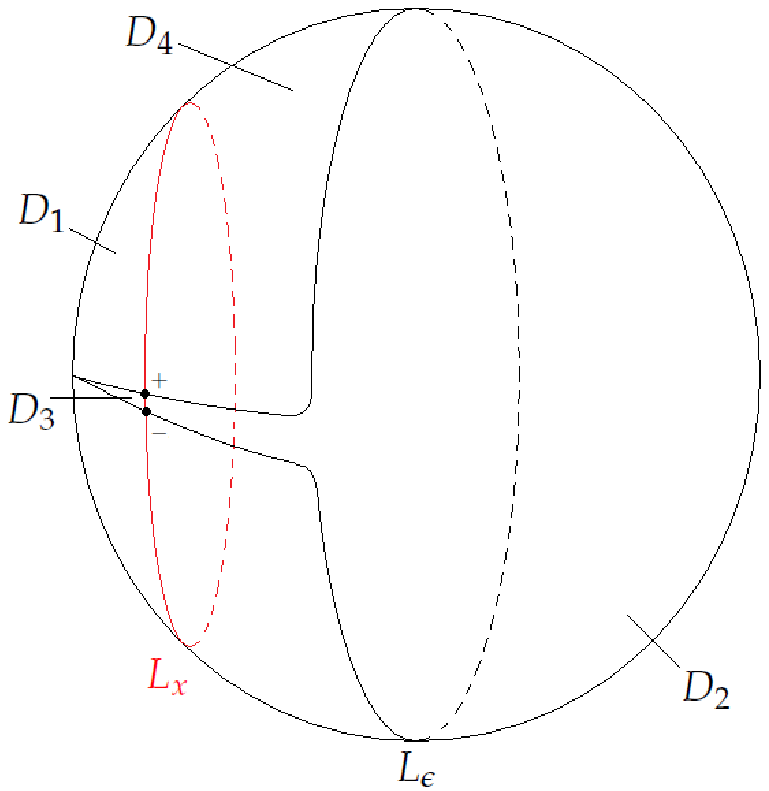}
\caption{}\label{discs}
\end{figure}

For $p=+, q=-$, the two holomorphic disks are given by $\varphi_1,\varphi_2:D^2\rightarrow X$ and their images in $\bP^1$ are denoted by $D_1$, $D_2$ in Figure \ref{discs}. $\partial_-[\varphi_1]=[\gamma_l]\in\pi_1(L_x;+,-)$ is the class of the major arc in $L_x$ going from $+$ to $-$, while $\partial_-[\varphi_2]=[\gamma_s]\in\pi_1(L_x;+,-)$ is the class of the minor arc in $L_x$ going from $+$ to $-$. As $L_\epsilon$ is deformed back to $L$, the areas of $D_1$ and $D_2$ tend to $x$ and $t/2$ respectively.

For $p=-, q=+$, the two holomorphic disks are given by $\varphi_3,\varphi_4:D^2\rightarrow X$, and their images in $\bP^1$ are denoted by $D_3$, $D_4$ in Figure \ref{discs}. $\partial_-[\varphi_3]=[\gamma_l]\in\pi_1(L_x;-,+)$ is the class of the minor arc in $L_x$ going from $-$ to $+$, while $\partial_-[\varphi_4]=[\gamma_s]\in\pi_1(L_x;-,+)$ is the class of the major arc in $L_x$ going from $-$ to $+$. As $L_\epsilon$ is deformed back to $L$, the areas of $D_3$ and $D_4$ tend to $0$ and $t/2-x$ respectively.

Now, using these holomorphic disks, we shall define a matrix-valued function $\Psi_{L,\nabla}$ over $U$ associated to the A-brane $(L,\nabla)$ as follows. For any $x\in U$ and for any $p,q\in L\cap L_x=\{+,-\}$, define
$$\Psi_{L,\nabla}^{p,q}(x,[\gamma])=\sum_{\substack{[\varphi]\in\pi_2(X;L,L_x;p,q),\\
\partial_+[\varphi]=[\gamma]}}\pm n([\varphi])\exp(-\int_{D^2}\varphi^*\omega_{\bar{X}})\textrm{hol}_{\nabla}([\gamma]),$$
for $[\gamma]\in\pi_1(L_x;p,q)$, where $n([\varphi])$ is the number of holomorphic disks $\varphi:(D^2,\partial D^2,-1,1)\rightarrow(X,L\cup L_x,p,q)$ representing the class $[\varphi]\in\pi_2(X;L,L_x;p,q)$, and the sign depends on the orientation of the moduli space of holomorphic disks with class $[\varphi]$ as discussed in \cite{FOOO06}.

Identify $\pi_1(L_x;p,q)$ with $\bZ$ for any $p,q\in L\cap L_x$. Then, as $x\in U$ varies, we get a function $\Psi_{L,\nabla}^{p,q}:U\times\bZ\rightarrow\bR$. By the above classification of disks, we have\footnote{In the two-dimensional case, there is a simple rule to determine the orientation of moduli spaces of holomorphic disks and hence the signs; see e.g. Chapter 8 in \cite{ABCDGKMSSW09}.}
\begin{eqnarray*}
\Psi_{L,\nabla}^{+,-}(x,v) & = & \left\{
\begin{array}{ll}
e^{-x} & \textrm{if $v=1$}\\
-e^{-t/2}=-\sqrt{q} & \textrm{if $v=0$}\\
0 & \textrm{otherwise,}
\end{array} \right.\\
\Psi_{L,\nabla}^{-,+}(x,v) & = & \left\{
\begin{array}{ll}
e^{-0}=1 & \textrm{if $v=0$}\\
-e^{-(t/2-x)}=-\sqrt{q}/e^{-x} & \textrm{if $v=-1$}\\
0 & \textrm{otherwise,}
\end{array} \right.
\end{eqnarray*}
and $\Psi_{L,\nabla}^{p,q}=0$ if $p=q$. Let $\Psi_{L,\nabla}$ be the matrix-valued function on $U\times\bZ$ given by
$$\Psi_{L,\nabla}=\left(
\begin{array}{cc}
0 & \Psi_{L,\nabla}^{+,-} \\
\Psi_{L,\nabla}^{-,+} & 0
\end{array} \right).$$
We regard $\Psi_{L,\nabla}$ as an object in the A-model of $X=\bP^1$.

From the perspective of the SYZ conjecture \cite{SYZ96} (see also \cite{Chan-Leung08a}, \cite{Chan-Leung08b}), the mirror manifold $\check{X}$ is constructed as the moduli space of pairs $(L_x,\nabla_y)$, where $L_x$ is a Lagrangian torus fiber of the moment map $\mu:X\rightarrow P$ and $\nabla_y$ ($y\in\bR/2\pi\bZ$) is a flat $U(1)$-connection on the trivial line bundle $\underline{\bC}$ over $L_x$. The complex coordinate on $\check{X}\subset\bC^*$ is given by $z=\exp(-x+iy)$. Now, the SYZ mirror transformation is nothing but fiberwise Fourier series (see \cite{Chan-Leung08a}, \cite{Chan-Leung08b}). Hence we have
\begin{thm} The SYZ mirror transformation of $\Psi_{L,\nabla}$ is
given by
$$\mathcal{F}(\Psi_{L,\nabla})=\left(
\begin{array}{cc}
0 & z-\sqrt{q} \\
1-\frac{\sqrt{q}}{z} & 0
\end{array} \right),$$
and this is equal to the matrix factorization $M_0$ corresponding to the structure sheaf of the point $e^{-x_0}=\sqrt{q}\in\check{X}$, which is mirror to the A-brane $(L_0,\nabla)$.
\end{thm}
Geometrically, the holomorphic disk which corresponds to the term $z$ (respectively $q/z$) in $W=z+q/z$ is cut into two holomorphic disks $D_1$ and $D_3$ (respectively, $D_2$ and $D_4$). These correspond to the factorizations of monomials
$$z=z\cdot1\textrm{ and }\frac{q}{z}=\sqrt{q}\cdot\frac{\sqrt{q}}{z}.$$
On the other hand, each of the holomorphic disks whose boundaries lie on the great circle $L_0$ (i.e. the half-spheres) is cut into a union of two discs: $D_1\cup D_4$ and $D_2\cup D_3$. These correspond to the factorizations
$$\sqrt{q}=1\cdot\sqrt{q}\textrm{ and }\sqrt{q}=\frac{\sqrt{q}}{z}\cdot z.$$

Furthermore, as mentioned in the introduction, $M_0$ should be mirror to the Floer differential $\mathfrak{m}_1=\mathfrak{m}_1((L_x,\nabla_y),(L,\nabla))$. Indeed, by definition (see e.g. \cite{FOOO06}), we have
\begin{eqnarray*}
\mathfrak{m}_1[+] & = & (e^{-A(D_1)}\textrm{hol}_{\nabla_y}(\partial_- D_1)-e^{-A(D_2)}\textrm{hol}_{\nabla_y}(\partial_- D_2))[-]\\
& = & (z-\sqrt{q})[-],\\
\mathfrak{m}_1[-] & = & (e^{-A(D_3)}\textrm{hol}_{\nabla_y}(\partial_-
D_3)-e^{-A(D_4)}\textrm{hol}_{\nabla_y}(\partial_- D_4))[+]\\
& = & (1-\frac{\sqrt{q}}{z})[+],
\end{eqnarray*}
where $A(D_i)$ denotes the symplectic area of the holomorphic disk $D_i$. Therefore, in matrix form, the Floer differential is given by
$$\mathfrak{m}_1=\left( \begin{array}{cc}
                0 & z-\sqrt{q} \\
                1-\frac{\sqrt{q}}{z} & 0
                \end{array} \right)=M_0.$$
\begin{nb}
We can deal with the case of $\bP^1\times\bP^1$ simply by taking the product of two copies of the above constructions.
\end{nb}

\section{The $\bP^2$ case}\label{sec3}

In this section, we shall try to imitate the construction of the last section to deal with the $\bP^2$ case. We will consider a Lagrangian torus fiber of the moment map of $\bP^2$ and try to find a Hamiltonian isotopic Lagrangian subspace which is a multi-section over some open subset of the moment polytope. We will then construct a matrix-valued function using a heuristic counting of holomorphic disks and show that the SYZ transformation of the function gives a matrix factorization which is mirror to the Lagrangian torus fiber that we start with.

\subsection{Mirror symmetry for $\bP^2$}

We equip $X=\bP^2$ with the toric K\"{a}hler form $\omega$ associated to the polytope given by
$$P=\{(x_1,x_2)\in\bR^2:x_1\geq0,\ x_2\geq0,\ x_1+x_2\leq t\},$$
where $t>0$; also let $\mu:X\rightarrow P$ be the moment map. Then, the mirror Landau-Ginzburg model is given by
\begin{eqnarray*}
\check{X}&=&\{(z_1,z_2)\in(\bC^*)^2:q<|z_1z_2|,|z_1|<1,|z_2|<1,\}\subset(\bC^*)^2,\\
W&=&z_1+z_2+\frac{q}{z_1z_2},
\end{eqnarray*}
where $z_1,z_2$ are coordinates on $(\bC^*)^2$ and $q=e^{-t}$.

We consider the Lagrangian torus fiber $L_0$ over the center of mass $x_0=(t/3,t/3)$ of the polytope $P$. This is the so-called \textit{Clifford torus} in $\bP^2$. As before, we equip $L_0$ with the trivial flat $U(1)$-connection $\nabla$ to give an A-brane $(L_0,\nabla)$ on $X=\bP^2$.

Applying the SYZ construction, the mirror B-brane of $(L_0,\nabla)$ should be given by the structure sheaf of the point $e^{-x_0}=(q^{1/3},q^{1/3})\in\check{X}$. A matrix factorization corresponding to this skyscraper sheaf is given by the $4\times4$ matrix
$$M_0=\left(\begin{array}{cccc}
0 & 0 & z_1-q^{1/3} & z_2-\frac{q^{2/3}}{z_1} \\
0 & 0 & -(1-\frac{q^{1/3}}{z_2}) & 1-\frac{q^{1/3}}{z_1}\\
1-\frac{q^{1/3}}{z_1} & -(z_2-\frac{q^{2/3}}{z_1}) & 0 & 0\\
1-\frac{q^{1/3}}{z_2} & z_1-q^{1/3} & 0 & 0
\end{array}\right).$$
It is straightforward to check that $M_0^2=(W(z_1,z_2)-W(q^{1/3},q^{1/3}))\textrm{Id}$.

As in the case of $\bP^1$, we shall work with symplectic Darboux coordinates $(x_1,x_2,u_1,u_2)$ on $\bP^2$, where $(x_1,x_2)\in\textrm{Int}(P)$ and $u_1,u_2\in\bR/2\pi\bZ$. In these coordinates, the complement of the toric boundary divisor $X\setminus D_\infty$ can be realized as the quotient $T^*\textrm{Int}(P)/\bZ^2$ of the cotangent bundle
of $\textrm{Int}(P)$ by the lattice $\bZ^2$ of locally constant 1-forms, and the symplectic form $\omega$ restricted to $X\setminus D_\infty$ is the canonical symplectic form, i.e. $\omega|_{X\setminus D_\infty}=dx_1\wedge du_1+dx_2\wedge du_2$.

\subsection{The deformed Lagrangian $L$ as a union of Lagrangian strata}

To construct the deformed Lagrangian subspace, we consider the map $$\tau^2:=\tau\times\tau:[0,3]^2\rightarrow\bC^2\subset\bP^2,$$
where $\tau:[0,3]\rightarrow(0,t/3)\times\bR/2\pi\bZ\subset\bC\subset\bP^1$ is the map defined in the previous section for the $\bP^1$ case, i.e.
\begin{equation*}
\tau(s)=\left\{\begin{array}{ll}
((1-s)t/3,0) & \textrm{for $0\leq s\leq1$;}\\
((s-1)t/3,0) & \textrm{for $1\leq s\leq2$;}\\
(t/3,2\pi(s-2)) & \textrm{for $2\leq s\leq3$.}\end{array} \right.
\end{equation*}
This defines a Lagrangian subspace $L$ in $\bP^2$, in the sense that $(\tau^2)^*\omega=0$. The image of $\tau^2$ is the union of three types of Lagrangian strata: a copy of the two-torus $L_0$, two $S^1\coprod S^1$-fibrations over the line segments $[0,t/3]\times\{t/3\}$ and $\{t/3\}\times[0,t/3]$ and 4 copies of the zero section of the moment $\mu$ over the square $[0,t/3]\times[0,t/3]$. See Figure \ref{CP^2}. In particular, $L$ is a multi-section over the open set $U:=(0,t/3)\times(0,t/3)\subset\textrm{Int}(P)$.

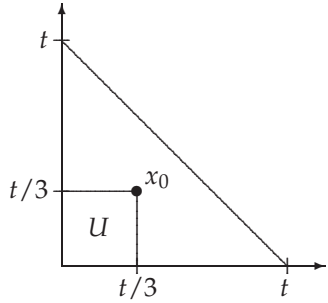
\begin{figure}[ht]
\setlength{\unitlength}{1mm}
\begin{picture}(100,40)
\put(40,2){\vector(0,1){35}} \put(40,2){\vector(1,0){35}} \curve(39,32, 41,32) \curve(70,1, 70,3) \put(37,31){$t$}
\put(69,-2){$t$} \curve(40,32, 70,2) \curve(50,2, 50,12) \curve(40,12, 50,12) \curve(39,12, 41,12) \curve(50,1, 50,3) \put(47.5,-2){$t/3$} \put(33,11){$t/3$} \put(49,11){$\bullet$} \put(51,13){$x_0$} \put(43.5,6){$U$}
\end{picture}
\caption{The polytope $P$ associated to $X=\bP^2$.}\label{CP^2}
\end{figure}

Note that $L$ can be deformed back to $L_0$ by deforming the map $\tau$ (and shrinking the open subset $U$). Hence, at least intuitively, we may regard $L$ to be in the same Lagrangian isotopy class as $L_0$. As both $L$ and $L_0$ sit inside $\bC^2$, we may further regard them as Hamiltonian isotopic to each other. Also, similar to the $\bP^1$ case, we can deform $L$ slightly to a Lagrangian torus $L_\epsilon$ embedded in $\bP^2$. For example, we can take the product of the deformed family of Lagrangian submanifolds $L_\epsilon\subset\bC\subset\bP^1$ (as shown in Figure \ref{CP^1}). This gives a family of Lagrangian submanifolds in $\bC^2$, also denoted by $L_\epsilon\subset\bC^2$, abusing notations. Let $V$ be a neighborhood of the line at infinity in $\bP^2$. Then $\bP^2\setminus V$ is symplectomorphic to $\bC^2\setminus V'$, where $V'\subset\bC^2$ is some open subset. Hence, our deformed Lagrangian subspace $L$ is indeed the limit of a family of Lagrangian tori $L_\epsilon\subset\bP^2$.

For any $x\in U$, denote by $L_x$ the Lagrangian torus fiber of the moment map $\mu$ over $x$. Then, $L$ and $L_x$ intersect at 4 points. Actually, $L$ and $L_x$ intersect at the same point (in the zero section) with multiplicity 4. But we shall think of these as 4 distinct intersection points. One way is to regard them as the intersection points of $L_\epsilon$ with $L_x$. Label these 4 points by $L\cap L_x=\{++,-+,+-,--\}$.

\subsection{Permissible pairs of paths}

Now, we want to find all the nontrivial holomorphic maps $\varphi:D^2\rightarrow X$ with $\varphi(\partial_+D^2)\subset L$, $\varphi(\partial_-D^2)\subset L_x$, $\varphi(-1)=p$ and $\varphi(1)=q$,\footnote{As in the $\bP^1$ case, we will also denote such a map by $\varphi:(D^2,\partial D^2,-1,1)\rightarrow(X,L\cup L_x,p,q)$.} for any pair of intersection points $p,q\in L\cap L_x=\{++,-+,+-,--\}$, and use these data to define the matrix-valued function $\Psi_{L,\nabla}$.

As we mentioned in the introduction, since $L$ is highly singular, it is very hard to classify these holomorphic disks. Hence, instead, we shall use a heuristic way to count the disks. This is done by sorting out the pairs of paths $\gamma_+:[0,1]\rightarrow L$, $\gamma_-:[0,1]\rightarrow L_x$ which, conjecturally, would form boundaries of holomorphic disks. In the next subsection, we will give an argument to justify our heuristic counting; to make this argument into a proof, however, we will need some sort of "gluing theorem" which is not available at the time of writing.

To start with, recall that the Lagrangian $L$ is given by the map $\tau^2:[0,3]^2\rightarrow\bP^2$. So we can regard a path in $L$ as a path in $[0,3]^2$. To fix notations, we subdivide $[0,3]^2$ into 9 regions, as shown in
Figure \ref{[0,3]^2} below. The Lagrangian subspace $L\subset\bP^2$ consists of 4 sections over $U\subset P$, which are parameterized by the regions labeled as $++,-+,+-,--$. Over each point of the line segments $\{t/3\}\times[0,t/3]$ and $[0,t/3]\times\{t/3\}$ in $P$, $L$ consists of two circles and these parts of $L$ are parameterized by the regions labeled as I, II, III, IV. The region labeled as 0 corresponds to the stratum of $L$ which is a copy of $L_0$.

\begin{figure}[ht]
\setlength{\unitlength}{1mm}
\begin{picture}(100,47)
\curve(28,0, 73,0) \curve(28,0, 28,45) \curve(28,45, 73,45)
\curve(73,0, 73,45) \curve(28,15, 73,15) \curve(28,30, 73,30)
\curve(43,0, 43,45) \curve(58,0, 58,45) \put(64,36.5){$0$}
\put(35,36.5){I} \put(49,36.5){II} \put(63,21.5){III}
\put(63,6.5){IV} \put(47.5,21.5){$++$} \put(32.5,21.5){$-+$}
\put(32.5,6.5){$--$} \put(47.5,6.5){$+-$}
\end{picture}
\caption{The 9 regions in $[0,3]^2$.}\label{[0,3]^2}
\end{figure}

Let $p,q\in L\cap L_x=\{++,-+,+-,--\}$ be two distinct intersection points, say $p=++$, $q=-+$. Then they are represented by two points in the corresponding regions in $[0,3]^2$. See the right hand side of Figure \ref{p,q}.

\begin{figure}[htp]
\setlength{\unitlength}{1mm}
\begin{picture}(100,55)
\curve(50,3, 95,3) \curve(50,3, 50,48) \curve(50,48, 95,48)
\curve(95,3, 95,48) \curve(50,18, 95,18) \curve(50,33, 95,33)
\curve(65,3, 65,48) \curve(80,3, 80,48) \put(68,26.5){$\bullet$}
\put(59.5,26.5){$\bullet$} \put(59.5,23.5){$q$} \put(68,23.5){$p$}
\put(70,-0.5){$L$} \put(10,23){\vector(1,0){30}}
\put(10,23){\vector(0,1){30}} \put(35,19){$(1,0)$}
\put(1,49){$(0,1)$} \curve(10,48, 35,48) \curve(35,23, 35,48)
\put(9.25,22.1){$\bullet$} \put(7,20){$p,q$} \put(22,19){$L_x$}
\put(10,12){\vector(0,1){6}} \put(5,8){$\bullet$}
\put(12,8){$\bullet$} \put(5,-2){$\bullet$} \put(12,-2){$\bullet$}
\put(-1,10){$-+$} \put(-1,-3){$--$} \put(14,10){$++$}
\put(14,-3){$+-$}
\end{picture}
\caption{}\label{p,q}
\end{figure}

While $++,-+,+-,--$ are the same point (the origin) in $L_x$ and in the image of the map $\tau^2:[0,3]^2\rightarrow\bP^2$, we shall keep in mind that they should be viewed as 4 distinct intersection points between $L_\epsilon$ and $L_x$. Their relative positions are, for instance, as shown in the left hand side of Figure \ref{p,q}.

Now, let $\bP^1_H$ (resp. $\bP^1_V$) be the line in $\bP^2$ which passes through the intersection points $L\cap L_x$ (remember that they are in fact the same point) and the torus-invariant point $[1:0:0]$ (resp. $[0:1:0]$) at infinity, where we have used the homogeneous coordinates on $\bP^2$. Then the intersections of $L$ and $L_x$ with $\bP^1_H$ (or $\bP^1_V$) resemble the situation of $\bP^1$ as shown in Figure \ref{discs}. To count holomorphic disks, we will need an identification of each pair of points $p,q$ (viewed inside $\bP^1_H$ or $\bP^1_V$) with either $+,-$ or $-,+$ (but not both) in Figure \ref{discs}. Without loss of generality, we identify both $++,-+$ and $+-,--$ with $+,-$ in $\bP^1_H$, and we identify both $++,+-$ and $-+,--$ with $+,-$ in $\bP^1_V$.

Recall that $L$ is a union of three types of Lagrangian strata: a copy of the two-torus $L_0$, two $S^1\coprod S^1$-fibrations over the line segments $[0,t/3]\times\{t/3\}$ and $\{t/3\}\times[0,t/3]$, and 4 copies of the zero section of the moment $\mu$ over the square $[0,t/3]\times[0,t/3]$. Accordingly, the image of a holomorphic disk $\varphi:(D^2,\partial D^2,-1,1)\rightarrow(X,L\cup L_x,p,q)$ can be broken down into several parts consisting of the following three types:
\begin{enumerate}
\item[$\bullet$] a disk in $\bP^1_H$ (or $\bP^1_V$) of the form $D_1$, $D_2$, $D_3$ or $D_4$ shown in Figure \ref{discs},
\item[$\bullet$] a disk whose boundary lies in one of the two $S^1\coprod S^1$-fibrations, and
\item[$\bullet$] a disk with boundary in $L_0$.
\end{enumerate}
Conversely, any combination of these three types of disks is a candidate for a holomorphic disk $\varphi:(D^2,\partial D^2,-1,1)\rightarrow(X,L\cup L_x,p,q)$. The following definition describes which combinations are allowed and will be counted.

Before we state the definition, recall that for each holomorphic disk $\varphi:D^2\to X$ which represents a class in $\pi_2(X;L,L_x;p,q)$, $\partial_+\varphi$ is a path in $L$ going from $q$ to $p$ and $\partial_-\varphi$ is a path in $L_x$ going from $p$ to $q$. We regard $\partial_+\varphi$ as a path in the domain $[0,3]^2$ of $\tau^2$. Let $\iota:[0,1]^2\rightarrow L_x\subset X$ be the map which defines the two-torus $L_x$ (see the left hand side of Figure \ref{p,q}). Then we regard $\partial_-\varphi$ as a path in the domain $[0,1]^2$ of $\iota$.
\begin{defn}\label{defn3.1}
Let $\bP^1_H$ and $\bP^1_V$ be as above. Let $p,q\in L\cap L_x=\{++,-+,+-,--\}$ be two distinct intersection points. Let $\gamma_-:[0,1]\rightarrow[0,1]^2$ and $\gamma_+:[0,1]\rightarrow[0,3]^2$ be two oriented, connected paths in $[0,1]^2$ and $[0,3]^2$ respectively such that $\iota\circ\gamma_-(0)=p$, $\iota\circ\gamma_-(1)=q$ and $\tau^2\circ\gamma_+(0)=q$, $\tau^2\circ\gamma_+(1)=p$. We say that the pair of paths $(\gamma_+,\gamma_-)$ is permissible if the following conditions are satisfied:
\begin{enumerate}
\item[1.]

$\gamma_-$ is simple path which is compatible with the relative positions of the points $\{++,-+,+-,--\}$ shown in the left hand side of Figure \ref{p,q}.; we also allow $\gamma_-$ to be an oriented constant path.

\item[2.]

$\gamma_+$ is a piecewise linear, simple path in $[0,3]^2$, which descends to a closed curve in $L$.

\item[3.]

The union $(\tau^2\circ\gamma_+)([0,1])\cup(\iota\circ\gamma_-)([0,1])\subset L\cup L_x$ of the images of the two paths is a union of the following three kinds of paths:
\begin{enumerate}
\item[(i)] The boundary of a disk in $\bP^1_H$ or $\bP^1_V$ of the form $D_1$, $D_3$ or $D_4$ shown in Figure \ref{discs}, or a pair of line segments with opposite orientations (bounding a disk with zero area) in $\bP^1_H$ or $\bP^1_V$,
\item[(ii)] A pair of line segments with opposite orientations (bounding a disk with zero area) in the $S^1\coprod S^1$-fibration over the line segment $\{t/3\}\times[0,t/3]\subset P$ contained in $L$.
\item[(iii)] The boundary of a Maslov index two disk in $(X,L_0)$ intersecting the line at infinity.
\end{enumerate}

\item[4.]

(Balancing condition) The union $(\tau^2\circ\gamma_+)([0,1])\cup(\iota\circ\gamma_-)([0,1])\subset L\cup L_x$ is either the boundary of a disk in $\bP^1_H$ or $\bP^1_V$, or the union of 3 disks, each from one of the three types (i), (ii), (iii) listed above.
\end{enumerate}
\end{defn}
Essentially, we are only allowing certain combinations of disk components as possible candidates for a holomorphic disk $\varphi:(D^2,\partial D^2,-1,1)\rightarrow(X,L\cup L_x,p,q)$. For a justification of this definition, see the next subsection.

In the meantime, our task is to classify all permissible pairs of paths $(\gamma_+,\gamma_-)$. This is given by the following theorem and diagrams.

\begin{figure}[htp]
\centering
\includegraphics[scale=1.5]{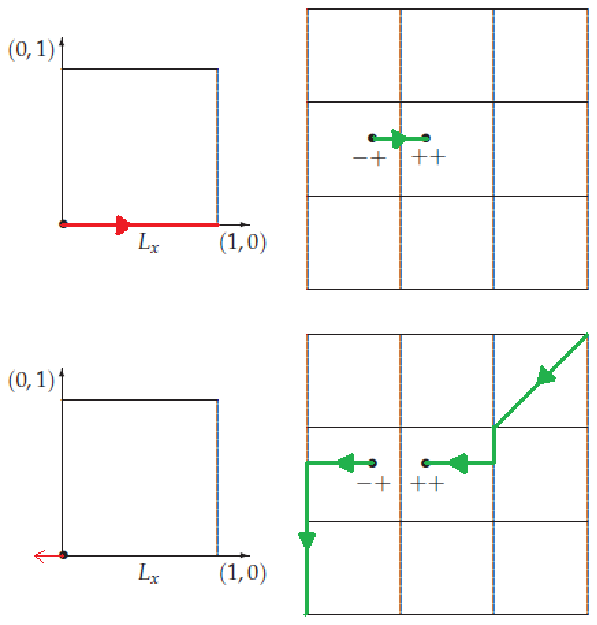}
\caption{$\gamma_+,\gamma_-$ for $p=++$, $q=-+$.}\label{++,-+}
\includegraphics[scale=1.5]{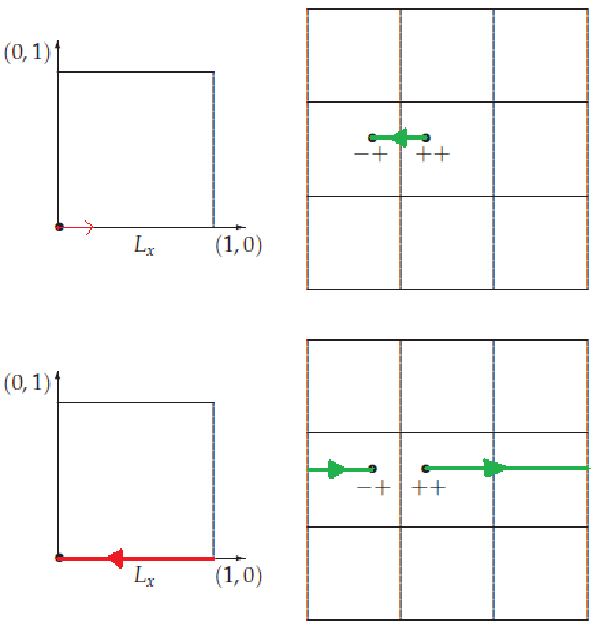}
\caption{$\gamma_+,\gamma_-$ for $p=-+$, $q=++$.}\label{-+,++}
\end{figure}
\begin{figure}[htp]
\centering
\includegraphics[scale=1.5]{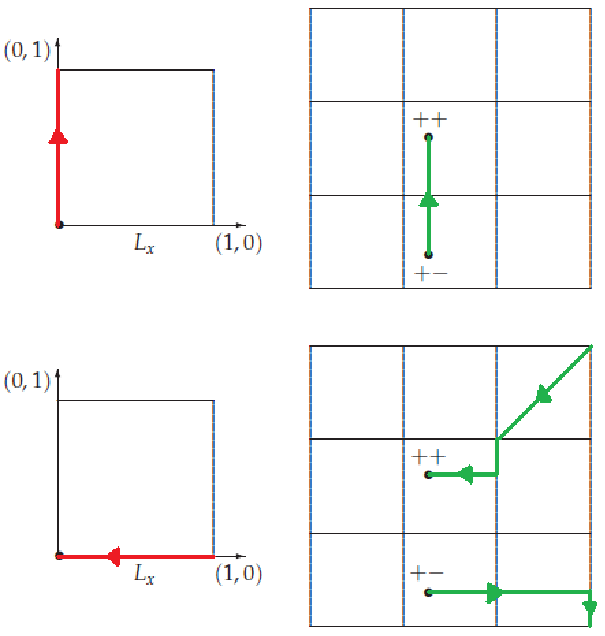}
\caption{$\gamma_+,\gamma_-$ for $p=++$, $q=+-$.}\label{++,+-}
\includegraphics[scale=1.5]{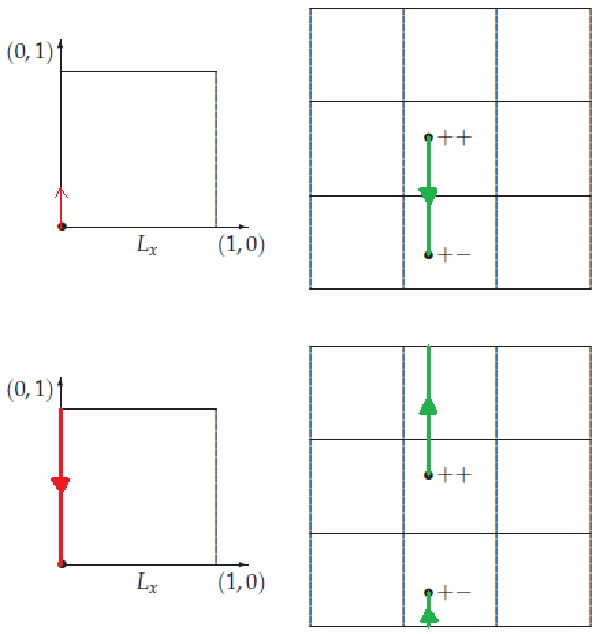}
\caption{$\gamma_+,\gamma_-$ for $p=+-$, $q=++$.}\label{+-,++}
\end{figure}

\begin{thm}\label{thm3.1}
Let
$(\gamma_+:[0,1]\rightarrow[0,3]^2,\gamma_-:[0,1]\rightarrow[0,1]^2)$ be a permissible pair of paths. Then, up to reparameterizations, $(\gamma_+,\gamma_-)$ belongs to the following list (see Figures~\ref{++,-+}, \ref{-+,++}, \ref{++,+-} and \ref{+-,++} for illustrations; the paths $\gamma_-,\gamma_+$ are drawn as thick red and green lines respectively, and a thin red arrow indicates oriented constant path):

For $p=++$, $q=-+$ (and by symmetry, for $p=+-$, $q=--$), either
\begin{enumerate}
\item[1.]

$\gamma_-$ is a straight line in the $(1,0)$ direction, and $\gamma_+$ is a horizontal line going to the right from $-+$ to $++$; or

\item[2.]

$\gamma_-$ is a constant path with orientation in the $(-1,0)$ direction, and $\gamma_+$ is a piecewise linear path going from $-+$ to $++$ as shown in Figure~\ref{++,-+}.
\end{enumerate}

For $p=-+$, $q=++$ (and by symmetry, for $p=--$, $q=+-$), either
\begin{enumerate}
\item[1.]

$\gamma_-$ is a constant path with orientation in the $(1,0)$ direction, and $\gamma_+$ is a horizontal line going to the left from $++$ to $-+$; or

\item[2.]

$\gamma_-$ is a straight line in the $(-1,0)$ direction, and $\gamma_+$ is a horizontal line going to the right starting from $++$ and ending at $-+$.
\end{enumerate}

For $p=++$, $q=+-$ (and by symmetry, for $p=-+$, $q=--$), either
\begin{enumerate}
\item[1.]

$\gamma_-$ is a straight line in the $(0,1)$ direction, and $\gamma_+$ is a vertical line going upwards from $+-$ to $++$; or

\item[2.]

$\gamma_-$ is a straight line in the $(-1,0)$ direction, and $\gamma_+$ is a piecewise linear path going from $+-$ to $++$ as shown in Figure~\ref{++,+-}.
\end{enumerate}

For $p=+-$, $q=++$ (and by symmetry, for $p=--$, $q=-+$), either
\begin{enumerate}
\item[1.]

$\gamma_-$ is a constant path with orientation in the $(0,1)$ direction, and $\gamma_+$ is a vertical line going downwards from $++$ to $+-$; or

\item[2.]

$\gamma_-$ is a straight line in the $(0,-1)$ direction. $\gamma_+$ is a vertical line going upwards starting from $++$ and ending at $+-$.
\end{enumerate}
Furthermore, each permissible pair of paths can be realized as the boundary of a union of disks (including disks with area zero).
\end{thm}
\begin{proof}
By condition 4, the disk bounded by $(\gamma_+,\gamma_-)$ consists of either one or three components.

In the one component case, we can make use of the results in the $\bP^1$ case to classify all possible cases. They are listed in Figures \ref{-+,++} and \ref{+-,++}, and the upper halves of Figures \ref{++,-+} and \ref{++,+-}.

Suppose that there are three components. Then the component sitting inside $\bP^1_H$ or $\bP^1_V$ is either a disk of the form $D_4$ in Figure \ref{discs} or a pair of line segments with opposite orientations. On the other hand, by condition 3(ii), this component must connect $p,q$ to the $S^1\coprod S^1$-fibration over the line segment $\{t/3\}\times[0,t/3]\subset P$ contained in $L$. Since $\gamma_+$ is connected, the component can only sit inside $\bP^1_H$. When the component is a disk of the form $D_4$ in Figure \ref{discs}, there is only one possible way to draw $\gamma_+$ and this is when $p=++,q=+-$ or $p=-+,q=--$. This is shown in the lower half of Figure \ref{++,+-}.
When the component is a pair of line segments with opposite orientations, there is again only one possible $\gamma_+$ and this is when $p=++,q=-+$ or $p=-+,q=--$. This is shown in the lower half of Figure \ref{++,-+}.
\end{proof}

The symplectic area of the disks bounded by the permissible pairs of paths listed in the above theorem are very easy to compute. The areas of those with one component are computed as in the $\bP^1$ case. For the two cases with three components, the area is given by $t/3$ for the case shown in the lower half of Figure \ref{++,-+} and $t/3+(t/3-x_1)=2t/3-x_1$ for the case shown in the lower half of Figure \ref{++,+-}.
\begin{defn}
Given a permissible pair $(\gamma_+,\gamma_-)$, we let $A(\gamma_+,\gamma_-)$ be the symplectic area of the (singular) disk bounded by $(\gamma_+,\gamma_-)$.
\end{defn}
Now, for any $p,q\in L_x\cap L$, we identify $\pi_1(L_x;p,q)$ with $\bZ^2$. We shall define a matrix-valued function $\Psi_{L,\nabla}$ on $U\times\bZ^2$, using the above classification of permissible pairs of paths. Each entry of $\Psi_{L,\nabla}$ is of the following form
$$\Psi_{L,\nabla}^{p,q}(x,[\gamma])=\sum_{\gamma_+}\pm\exp(-A(\gamma_+,\gamma))\textrm{hol}_\nabla(\gamma),$$
where $\gamma$ is a path in $L_x$ which represents a class $[\gamma]\in\pi_1(L_x;p,q)$ and the sum is over all paths $\gamma_+:[0,1]\rightarrow[0,3]^2$ such that the pair $(\gamma_+,\gamma)$ is permissible, for a distinct pair of intersection points $p,q$. More precisely, in view of Theorem \ref{thm3.1}, we set\footnote{Unfortunately, there is no simple rule to determine the orientation of moduli spaces of holomorphic disks in higher dimensional cases, and we have to assign the signs \textit{by hand}.}
\begin{eqnarray*}
\Psi_{L,\nabla}^{++,-+}(x,v) & = & \left\{\begin{array}{ll}
e^{-x_1} & \textrm{if $v=(1,0)$}\\
-e^{-t/3}=-q^{1/3} & \textrm{if $v=(0,0)$}\\
0 & \textrm{otherwise;} \end{array} \right.\\
\Psi_{L,\nabla}^{++,+-}(x,v) & = & \left\{\begin{array}{ll}
e^{-x_2} & \textrm{if $v=(0,1)$}\\
-e^{-(2t/3-x_1)}=-\frac{q^{2/3}}{e^{-x_1}} & \textrm{if $v=(-1,0)$}\\
0 & \textrm{otherwise;} \end{array} \right.\\
\Psi_{L,\nabla}^{--,-+}(x,v) & = & \left\{\begin{array}{ll}
-e^{-0}=-1 & \textrm{if $v=(0,0)$}\\
e^{-(t/3-x_2)}=\frac{q^{1/3}}{e^{-x_2}} & \textrm{if $v=(0,-1)$}\\
0 & \textrm{otherwise;} \end{array} \right.\\
\Psi_{L,\nabla}^{--,+-}(x,v) & = & \left\{\begin{array}{ll}
e^{-0}=1 & \textrm{if $v=(0,0)$}\\
-e^{-(t/3-x_1)}=-\frac{q^{1/3}}{e^{-x_1}} & \textrm{if $v=(-1,0)$}\\
0 & \textrm{otherwise;} \end{array} \right.\\
\Psi_{L,\nabla}^{-+,++}(x,v) & = & \left\{\begin{array}{ll}
e^{-0}=1 & \textrm{if $v=(0,0)$}\\
-e^{-(t/3-x_1)}=-\frac{q^{1/3}}{e^{-x_1}} & \textrm{if $v=(-1,0)$}\\
0 & \textrm{otherwise;} \end{array} \right.\\
\Psi_{L,\nabla}^{-+,--}(x,v) & = & \left\{\begin{array}{ll}
-e^{-x_2} & \textrm{if $v=(0,1)$}\\
e^{-(2t/3-x_1)}=\frac{q^{2/3}}{e^{-x_1}} & \textrm{if $v=(-1,0)$}\\
0 & \textrm{otherwise;} \end{array} \right.\\
\Psi_{L,\nabla}^{+-,++}(x,v) & = & \left\{\begin{array}{ll}
e^{-0}=1 & \textrm{if $v=(0,0)$}\\
-e^{-(t/3-x_2)}=-\frac{q^{1/3}}{e^{-x_2}} & \textrm{if $v=(0,-1)$}\\
0 & \textrm{otherwise;} \end{array} \right.\\
\Psi_{L,\nabla}^{+-,--}(x,v) & = & \left\{\begin{array}{ll}
e^{-x_1} & \textrm{if $v=(1,0)$}\\
-e^{-t/3}=-q^{1/3} & \textrm{if $v=(0,0)$}\\
0 & \textrm{otherwise,} \end{array} \right.
\end{eqnarray*}
and the matrix-valued function $\Psi_{L,\nabla}$ to be given by
$$\Psi_{L,\nabla}=\left(\begin{array}{cccc}
0 & 0 & \Psi_{L,\nabla}^{++,-+} & \Psi_{L,\nabla}^{++,+-}\\
0 & 0 & \Psi_{L,\nabla}^{--,-+} & \Psi_{L,\nabla}^{--,+-}\\
\Psi_{L,\nabla}^{-+,++} & \Psi_{L,\nabla}^{-+,--} & 0 & 0\\
\Psi_{L,\nabla}^{+-,++} & \Psi_{L,\nabla}^{+-,--} & 0 & 0
\end{array} \right).$$

As in the $\bP^1$ case, the mirror manifold $\check{X}$ can be constructed as the moduli space of pairs $(L_x,\nabla_y)$, where $L_x$ ($x=(x_1,x_2)\in\textrm{Int}(P)$) is a Lagrangian torus fiber of the moment map $\mu:X\rightarrow P$ and $\nabla_y$ ($y=(y_1,y_2)\in(\bR/2\pi\bZ)^2$) is a flat $U(1)$-connection on the trivial line bundle $\underline{\bC}$ over $L_x$. Also, the coordinates on $\check{X}\subset(\bC^*)^2$ are given by $z_1=\exp(-x_1+iy_1)$, $z_2=\exp(-x_2+iy_2)$. The SYZ mirror transformation is again taking fiberwise Fourier series (see \cite{Chan-Leung08a}, \cite{Chan-Leung08b}). Hence we finally comes to
\begin{thm}
The SYZ mirror transformation of $\Psi_{L,\nabla}$ is given by
$$\mathcal{F}(\Psi_{L,\nabla})=\left(\begin{array}{cccc}
0 & 0 & z_1-q^{1/3} & z_2-\frac{q^{2/3}}{z_1} \\
0 & 0 & -(1-\frac{q^{1/3}}{z_2}) & 1-\frac{q^{1/3}}{z_1}\\
1-\frac{q^{1/3}}{z_1} & -(z_2-\frac{q^{2/3}}{z_1}) & 0 & 0\\
1-\frac{q^{1/3}}{z_2} & z_1-q^{1/3} & 0 & 0
\end{array}\right),$$
and this is equal to the matrix factorization $M_0$ corresponding to the structure sheaf of the point $e^{-x_0}=(q^{1/3},q^{1/3})\in\check{X}$, which is mirror to the A-brane $(L_0,\nabla)$.
\end{thm}
\begin{nb}
We can try to interpret the formula $M_0^2=(W(z_1,z_2)-W(q^{1/3},q^{1/3}))\textrm{Id}$ geometrically as cutting each Maslov index two holomorphic disk which corresponds to a term in
$W(z_1,z_2)$ or $W(q^{1/3},q^{1/3})$ into disks $\varphi:(D^2,\partial D^2,-1,1)\rightarrow(X,L\cup L_x,p,q)$. However, the cutting of disks in the $\bP^2$ case is not as straightforward as in the $\bP^1$ case. For example, the
factorization
$$\frac{q}{z_1z_2}=\frac{q^{2/3}}{z_1}\cdot\frac{q^{1/3}}{z_2}$$
is not directly corresponding to cutting the Maslov index two holomorphic disk which corresponds to the term $q/z_1z_2$ in $W(z_1,z_2)$ into two disks which represent classes in $\pi_2(X;L,L_x;p,q)$. Intuitively, deforming $L_0$ to $L$ splits this Maslov index two disk into two disks which correspond to the two factors $q^{2/3}/z_1$ and $q^{1/3}/z_2$ in the above factorization. Their boundaries are shown in the lower halves of Figures \ref{++,+-} and \ref{+-,++}.
\end{nb}

\subsection{A heuristic argument}

As we have mentioned in the introduction, it is hard to classify the holomorphic disks. However, we do believe that when $L$ is smoothed out to $L_\epsilon$, Definition \ref{defn3.1} would give the correct restrictions on a pair of paths $(\gamma_+,\gamma_-)$ to form the boundary of a holomorphic disk.
\begin{conj}\label{conjecture}
The set of permissible pairs of paths $(\gamma_+,\gamma_-)$ given in Theorem \ref{thm3.1} is in a bijective correspondence with the set of isomorphism classes of holomorphic disks $\{\varphi:(D^2,\partial D^2,-1,1)\rightarrow(X,L_\epsilon\cup L_x,p,q):p,q\in L_\epsilon\cap L_x\}$. Consequently, the value of the matrix factorization $M_0$ at a point $(z_1=\exp(-x_1+iy_1),z_2=\exp(-x_2+iy_2))\in\check{X}$, where $x=(x_1,x_2)\in U'$ and $y=(y_1,y_2)\in(\bR/2\pi\bZ)^2$, coincides with the Floer differential $\mathfrak{m}_1$ (computed over $\bC$) for the pair of A-branes $(L_\epsilon,\nabla)$ and $(L_x,\nabla_y)$.
\end{conj}

We are still far away from proving this conjecture. Nevertheless, we shall try to give a justification of Definition \ref{defn3.1} and Conjecture \ref{conjecture} in the following.

First of all, recall that the image of $L$ consists of overlapping $S^1$-fibrations (over the line segments $[0,t/3]\times\{t/3\}$ and $\{t/3\}\times[0,t/3]$) and zero sections (over $U$). Hence, if $L'$ is a certain smoothing of $L$ as described in the above conjecture, then we expect that some parts of a holomorphic disk $\varphi:(D^2,\partial D^2,-1,1)\rightarrow(X,L'\cup L_x,p,q)$ would be squeezed into a line segment when $L'$ degenerates to $L$. This is why we allow the boundary of a disk with zero area, realized as a pair of line segments with opposite orientations, to be a component of a permissible pair. This also explains why $\gamma_-$ can be a constant path.

Second, as $L$ is being smoothed out to give $L_\epsilon$, we assume that the intersection points $L\cap L_x=\{++,-+,+-,--\}$ are moving away from each other and their relative positions are as shown in the left hand side of Figure \ref{p,q}. Hence, we need to impose condition 1 in Definition \ref{defn3.1}.

On the other hand, since $L$ is a union of Lagrangian strata, a holomorphic disk $\varphi:(D^2,\partial D^2,-1,1)\rightarrow(X,L\cup L_x,p,q)$ is at best piecewise smooth. So we assume that $\gamma_+$ is piecewise smooth in condition 2 of Definition \ref{defn3.1}. That we require that it is piecewise linear is because, as shown in the $\bP^1$ case, the boundary of a holomorphic disk is linear and goes in some specific directions.

Now, since we expect that a holomorphic disk $\varphi:(D^2,\partial D^2,-1,1)\rightarrow(X,L\cup L_x,p,q)$ is piecewise smooth, it is natural to break such a disk into several components, each is bounded by a stratum of $L$ (and sometimes together with $L_x$). If the stratum is the copy of $L_0$, then the only natural candidates are the three Maslov index two disks with boundary in $L_0$; and in view of the balancing condition (condition 4) discussed below, we expect that only the disk with nontrivial intersection with the line at infinity will occur. So we impose condition 3(iii). When the stratum is the 4 copies of the zero section, we expect that a disk component would be contained in either $\bP^1_H$ or $\bP^1_V$. And as such, they should be classified in the same way as the $\bP^1$ case shown in Figure \ref{discs}. Moreover, a disk of the form $D_2$ would not occur since it is of Maslov index 4 as a disk in $\bP^2$. This explains condition 3(i).

If the stratum is one of the $S^1\coprod S^1$-fibrations, then we expect that there is no nontrivial disk whose boundary is contained entirely in this stratum. This is because the $S^1$'s can either bound a cylinder, in which case the other boundary must lie in $L_x$, or a disk with Maslov index 4 which should not be counted anyway. Therefore, the only possibility is a disk with zero area, or a pair of line segments with opposite orientations. This explains part of the reason why we impose condition 3(ii).

\begin{figure}[ht]
\setlength{\unitlength}{1mm}
\begin{picture}(100,40)
\put(50,3){\vector(-1,0){10}} \put(50,3){\vector(1,0){10}}
\put(49,-1){$\epsilon^2$} \put(38,2){$\bullet$} \put(31,1.5){$--$}
\put(60,2){$\bullet$} \put(62,1.5){$+-$} \put(38,36){$\bullet$}
\put(31,36.5){$-+$} \put(60,36){$\bullet$} \put(62,36.5){$++$}
\put(39,20){\vector(0,-1){16}} \put(39,20){\vector(0,1){16}}
\put(37,18){$\epsilon$}
\end{picture}
\caption{}\label{epsilon}
\end{figure}

Here comes a subtle point, namely, why we allow just one of the $S^1\coprod S^1$-fibrations in $L$, but not both? The reason is that when $L$ is being smoothed out and the intersection points $\{++,-+,+-,--\}$ are moving away from each other, we expect that there is a choice of the relative moving speeds of the points. For example, by imposing condition 3(ii) of Definition \ref{defn3.1} in the way we did, we have implicitly chosen the smoothing so that the distance between $++$ and $-+$ (or $+-$ and $--$) is of the order $O(\epsilon^2)$, while the distance between $++$ and $+-$ (or $-+$ and $--$) is of the order $O(\epsilon)$. See Figure \ref{epsilon} for an illustration. Intuitively, this choice means that a pair of line segments contained in the regions $++,--$ (or $-+,+-$) would bound a holomorphic disk only when they are vertical, but not horizontal. This explains why we have condition 3(ii).

If we instead choose the smoothing so that the distance between $++$ and $-+$ (or $+-$ and $--$) is of the order $O(\epsilon)$, and the distance between $++$ and $+-$ (or $-+$ and $--$) is of the order $O(\epsilon^2)$, then we would be allowing only pairs of horizontal, but not vertical, line segments in $++,--$ (or $-+,+-$) to bound a disk. In this case, for example, the following $\gamma_+$ (Figure \ref{notpermissible}) is allowed, but not the one shown in the lower half of Figure \ref{++,+-}.

\begin{figure}[htp]
\centering
\includegraphics[scale=1.5]{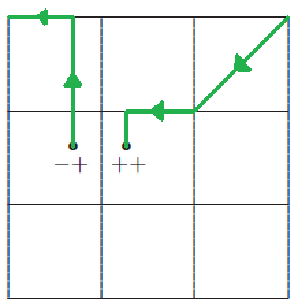}
\caption{}\label{notpermissible}
\end{figure}

Finally, we come to the balancing condition (condition 4), which is another key condition in the definition of a permissible pair of paths. In general, when we try to deform a singular holomorphic disk into a smooth holomorphic disk, there are some necessary conditions (or integrability conditions). In our situation, we expect that we should have the following condition: For three disk components whose moment map images lie in three directions, say $v_1,v_2,v_3$, the union is smoothable if $v_1+v_2+v_3=0$. This is based on the following fact: Suppose that $\gamma_1$, $\gamma_2$ and $\gamma_3$ are gradient flow line segments of three functions $f_1$, $f_2$ and $f_3$ on $\bR^2$ respectively, and they form a tree $T\subset\bR^2$ with a 3-valent vertex. Let $\Gamma_1$, $\Gamma_2$ and $\Gamma_3$ be the Lagrangian submanifolds in $T^*\bR^2$ given by the graphs of the exact 1-forms $df_1$, $df_2$ and $df_3$ respectively. Then the tree can be deformed to a holomorphic disk bounded by the Lagrangians $\Gamma_1$, $\Gamma_2$, $\Gamma_3$ if and only if $f_1+f_2+f_3=0$ (see e.g. Fukaya-Oh \cite{FO97}).

Now, the moment map image of a disk component allowed by condition 3(i) lies in either the $(1,0)$ or $(0,1)$ directions and that of condition 3(ii) lies in the $(0,1)$ direction. Hence, to get it balanced, we must have a disk whose moment map image lies in the direction $(-1,-1)$. This is why we set condition 3(iii). Moreover, a holomorphic disk $\varphi:(D^2,\partial D^2,-1,1)\rightarrow(X,L\cup L_x,p,q)$ should then have either one component or three components. This explains why we have condition 4 in Definition \ref{defn3.1}.

This concludes our heuristic reasoning and justification for Definition \ref{defn3.1} and Conjecture \ref{conjecture}. As we mention before, to make this informal argument into a real proof, we will (at least) need some sort of gluing theorems, which are not available at the moment.

\section{Comments: Other del Pezzo surfaces}

We expect that we can play the same game and extend Theorem \ref{main} to other toric del Pezzo surfaces, and even higher dimensional toric Fano manifolds. In particular, by essentially the same constructions and arguments, we can deal with the blowups of $\bP^2$ at one and two points. We leave this as an exercise to the reader.

However, for the blowup of $\bP^2$ at three points, things become more subtle. Indeed, we cannot get the correct matrix factorization by directly applying the constructions in this paper. We do not know why this is so, but it is possibly related to the smoothability of the deformed Lagrangian subspace $L$.

To deal with this case and hence obtain a unified treatment for all toric del Pezzo surfaces, we would need to choose a better deformed Lagrangian subspace $L$ and try to classify the holomorphic disks. We shall leave this and a proof of Conjecture \ref{conjecture} to future research.

\end{document}